\documentclass[12pt,reqno]{amsart}

\usepackage{amssymb,thmtools}

\usepackage[shortlabels]{enumitem}

\usepackage{microtype}

\usepackage[alphabetic, nobysame]{amsrefs}

\usepackage{hyperref}
\hypersetup{
	pdfstartview={XYZ null null 1.00}, 
	pdfpagemode=UseNone, 
	colorlinks,
	breaklinks, 
	linkcolor=blue,
	urlcolor=blue, 
	anchorcolor=blue,
	citecolor=blue
}

\usepackage[capitalise]{cleveref}

\usepackage{geometry}
\usepackage{comment}

\usepackage{xspace}
\usepackage{ bbold }
\usepackage{graphicx}

\theoremstyle{plain}

\newtheorem{lemma}[equation]{Lemma}

\crefname{prop}{Proposition}{Propositions}

\newtheorem{theorem}[equation]{Theorem}

\crefname{obs}{Observation}{Observations}

\crefname{cor}{Corollary}{Corollaries}

\makeatletter
\def\@empty{}
\def\ifemptycredit#1{%
	\def\tmp{#1}%
	\ifx\tmp\@empty%
	\else%
	{~(#1)}%
	\fi%
}
\makeatother

\newenvironment{namedthm*}[2][]{
\par\noindent \textbf{#2}\ifemptycredit{#1}\textbf{.}\itshape\xspace
}{}

\theoremstyle{definition}

\crefname{defn}{Definition}{Definitions}
\newtheorem{defn}[equation]{Definition}

\crefname{example}{Example}{Examples}

\theoremstyle{remark}

\newtheorem*{acknowledge}{Acknowledgements}

\crefname{remark}{Remark}{Remarks}

\crefname{claim}{Claim}{Claims}

\newtheorem*{claim*}{Claim}

\crefname{assumption}{Assumption}{Assumptions}

\declaretheoremstyle[
spaceabove=\topsep, 
spacebelow=6pt,
headfont=\normalfont\itshape,
notefont=\normalfont, notebraces={(}{)},
bodyfont=\normalfont,
postheadspace=4pt,
qed=\mbox{\smaller[4]$\boxtimes$}
]{claimproofstyle}
\declaretheorem[name={Proof of Claim}, style=claimproofstyle, unnumbered]{pf}

\let\#\mathbb
\let\@\mathcal

\newcommand{\N}{\mathbb{N}}

\newcommand{\R}{\mathbb{R}}
\newcommand{\Z}{\mathbb{Z}}

\newcommand{\e}{\varepsilon}

\renewcommand{\phi}{\varphi}

\newcommand*{\defeq}{\mathrel{\vcenter{\baselineskip0.5ex \lineskiplimit0pt \hbox{\scriptsize.}\hbox{\scriptsize.}}}=}

\newcommand{\set}[1]{\left\{ #1 \right\}}

\newcommand{\rest}[1]{\mathord{\downharpoonright_{#1}}}

\newcommand{\fsup}[1][f]{{#1}^*}
\newcommand{\finf}[1][f]{{#1}_*}

\title{A tiling property for actions of amenable groups along Tempelman F{\o}lner sequences}
\author{Jonathan Boretsky}
\author{Jenna Zomback}
\thanks{The first author conducted research through the summer research USRA program at McGill University
under the supervision of Marcin Sabok supported by NSERC Discovery Grant RGPIN 2015-03738.}

\date{\today}

\begin{document}

\vspace{-0.5cm}

\begin{abstract}
    We show that a certain tiling property (which directly implies the pointwise ergodic theorem)
    holds for pmp actions of amenable groups along increasing Tempelman F{\o}lner sequences, thus providing
	a short and combinatorial proof of the corresponding pointwise ergodic theorem. 
\end{abstract}

\vspace{-1cm}
\maketitle

\section{Introduction}
For a group $\Gamma$ acting on a probability space $X$ and a sequence $(F_n)$ of finite subsets of $\Gamma$, the pointwise ergodic property for $\Gamma$ along $(F_n)$ says that the action of $\Gamma$ is ergodic if and only if for every $L^1$ function $f$ on $X$, the integral (the global average) of $f$ over $X$ is equal to the limit of the averages of $f$ over $F_n\cdot x$ (the pointwise average) for almost every $x\in X$. The classical ergodic theorem, due to G. D. Birkhoff in 1931 \cite{Bir}, says that probability measure preserving (pmp) actions of $\Z$ along the sequence $[0,n)$ have the pointwise ergodic property. In 2001, E. Lindenstrauss proved that actions of amenable groups along tempered F\o lner sequences have the pointwise ergodic property \cite{Lin}.

A. Tserunyan in \cite{Tse} gives a short, combinatorial proof of the classical pointwise ergodic theorem (for $\Z$) by reducing it to showing that the following tiling property holds for pmp actions of $\Gamma=\Z$ along the intervals $F_n=[0,n)$:

\begin{defn}[Tiling Property]\label{Tiling}
We say that a pmp action of a countable group $\Gamma$ on a standard probability space $(X,\mu)$ has the \textit{tiling property} along a sequence $(F_n)$ of finite subsets of $\Gamma$ if for any pointwise increasing sequence of measurable functions $\ell_n : X \to \N$, $n \in \N$, and $\epsilon > 0$, there are arbitrarily large finite subsets $T \subseteq \Gamma$ such that for a set of points $x$ of measure at least $(1-\epsilon)$, $T\cdot x$ can be covered up to $\epsilon$ fraction by disjoint sets of the form $F_{\ell_i(y)}\cdot y$ (where $T\cdot x$ and $F_{\ell_i(y)}\cdot y$ are treated as multisets if the action of $\Gamma$ is not free). 
\end{defn}

It is also implicit in \cite{Tse} that the tiling property implies the pointwise ergodic theorem for any pmp group action (see \cref{sec_implication} for a proof). This implication distills out the analytic part from the proofs of pointwise ergodic theorems, reducing them to combinatorial (finitary) tiling problems.
Another proof of the ergodic theorem for $\Z$ revolving around the same idea was given in \cite{Keane-Petersen:ergodic_thm}. 

In this paper, 
we prove that the tiling property holds for pmp actions of amenable groups along increasing Tempelman F{\o}lner sequences $(F_n)$ by finding Vitali covers with F{\o}lner tiles on multiple scales. As a consequence, we prove the corresponding pointwise ergodic theorem:

\begin{theorem}[Pointwise ergodic]\label{ptwise_ergodic}
Fix a pmp action of an amenable group $\Gamma$ on a standard probability space $(X,\mu)$ and an increasing Tempelman F\o lner sequence $(F_n)$. Then the action of $\Gamma$ on $X$ is ergodic if and only if for every $f\in L^1(X,\mu)$, 
$$\lim_n A_f[F_n\cdot x]=\int_X f(x)d\mu(x)\text{ a.e.}$$
where $A_f[F_n\cdot x]\defeq \frac{1}{|F_n|}\sum_{\gamma\in F_n}f(\gamma\cdot x)$.
\end{theorem}

Although this is less general than Lindenstrauss's theorem, our proof is shorter and offers the advantage that the methods used are more elementary and finitary. 

Many people have shown this result for increasing Tempelman F{\o}lner sequences. The shortest proof of \cref{ptwise_ergodic} that the authors are aware of is given in \cite{OW}, which uses a Vitali covering lemma along with basic functional analysis: a function $f\in L^1(X,\mu)$ is approximated by functions for which the ergodic theorem holds trivially, and the error is controlled by applying the Vitali covering lemma.  Other proofs include \cite{Eme} and \cite{Tem}, which also use a Vitali covering lemma along with analysis. 
However, none of these proofs yield the tiling property described above, and hence they do not take advantage of the abstract implication of the corresponding pointwise ergodic theorem. 

\subsubsection*{A word on the proof of the tiling property} Compared to \cite{Tse}, the tiling property is much harder to establish for general increasing Tempelman F\o lner sequences. For example, tiling $\Z^d$ with boxes of different given sizes for each center is harder than tiling $\Z$ with intervals. 
The key idea in mitigating this difficulty is to iterate the Vitali covering lemma to find covers on multiple scales. We essentially zoom very far out, cover some constant fraction of the space with large sets (this fraction comes from the Tempelman condition and is independent of how far we've zoomed out), and then zoom in on the spots we miss, and fill those in as best we can with smaller sets, and so on and so forth. Since we cover a constant fraction on each scale, if we zoom out far enough at the beginning, once we zoom all the way back in, we will have covered nearly the whole space. 

\subsubsection*{Organization} In \cref{sec_def}, we provide the necessary definitions and notation that will be used throughout the paper. In \cref{sec_implication}, we give an explicit proof, due to Tserunyan, that \cref{Tiling} implies the pointwise ergodic property for any pmp group action. In \cref{differentsizes}, we establish the tiling property for pmp actions of amenable groups along increasing Tempelman F\o lner sequences, which then directly implies the corresponding pointwise ergodic property.

\begin{acknowledge}
The authors would like to thank their advisors, Marcin Sabok and Anush Tserunyan, for their guidance, suggestions, and support. Many thanks as well to Benjamin Weiss for pointing out the proof of Theorem 2 given in \cite{OW}, and to Alexander Kechris for useful suggestions.
\end{acknowledge}

\section{Definitions and notation}\label{sec_def}
Let $(X,\mu)$ be a standard probability space, and a function $f:X\rightarrow\R$. For a finite set $A\subseteq X$, define the \textit{average} of $f$ over $A$, $A_f[A]\defeq\frac{1}{|A|}\sum_{x\in A}f(x)$. For a finite equivalence relation $F$ on $X$, define $A_f[F](x)\defeq A_f[[x]_F]$.
Given a group $\Gamma$ and a finite set $R$, define the $R$-\textit{boundary} of a set $S$, denoted $\partial_R S$, to be the set of points $s$ for which $Rs\cap S\neq \emptyset$ and $R_s\cap S^c\neq \emptyset$. 

A sequence $(F_n)_{n\in\N}$ of finite subsets of $\Gamma$ is a \textit{F\o lner} sequence if $\Gamma=\bigcup_n F_n$ and $\lim_n\frac{|\partial_R F_n|}{|F_n|}=0$ for all finite sets $R$. A group $\Gamma$ is called \textit{amenable} if it admits a F\o lner sequence. For this paper, we will assume that our F\o lner sequences are increasing. 

Given a F\o lner sequence $(F_n)$, we say $(F_n)$ is \textit{tempered} if there is some natural number $C$ such that for all $n$,
$$\left|\bigcup_{k<n}F_k^{-1}F_n\right|\leq C|F_n|$$
and \textit{Tempelman} if there is $C$ such that for all $n$,
$$\left|\bigcup_{k\leq n}F_k^{-1}F_n\right|\leq C|F_n|$$
in the latter case, we'll call the smallest such $C$ the \textit{Tempelman constant} of $(F_n)$. Note that any Tempelman F\o lner sequence is tempered.

Every amenable group has a tempered F{\o}lner sequence (in fact, every F{\o}lner sequence has a tempered subsequence). In \cite{Lin}*{Example 4.2}, an example is given of an amenable group without a Tempelman F{\o}lner sequence. However, \cite{Hoc}*{Theorem 3.4} gives a sufficient condition for the existence of a Tempelman F{\o}lner sequence:

\begin{theorem}[Hochman 2007]
If for a countable, abelian, amenable group $G$, we have
$$r(G)=\sup\{n\in\mathbb{N}: G\text{ contains a subgroup isomorphic to  } \Z^n\}<\infty$$
then $G$ possesses at least one Tempelman F{\o}lner sequence.
\end{theorem}

\section{The tiling property implies the pointwise ergodic theorem}\label{sec_implication}

The following result is implicitly stated in \cite{Tse} and was explained by Tserunyan to the second author.
\begin{theorem}[Tserunyan]\label{lemma_tserunyan}
Assume $\Gamma$ has the tiling property along a sequence of finite subsets $(F_n)$. Then for any pmp action of $\Gamma$ on a standard probability space $(X,\mu)$, the action of $\Gamma$ on $X$ is ergodic if and only if for every $f\in L^1(X,\mu)$, 
$$\lim_n A_f[F_n\cdot x]=\int_X f(x)d\mu(x)\text{ a.e.}$$
where $A_f[F_n\cdot x]\defeq \frac{1}{|F_n|}\sum_{\gamma\in F_n}f(\gamma\cdot x)$.
\end{theorem}

\begin{proof}
By replacing $f$ with $f - \int f$, we may assume without loss of generality that $\int f = 0$. We will show that $\fsup \defeq \limsup_{n \to \infty} A_f[F_n\cdot x] \leq 0$ a.e., and an analogous argument shows $\finf \defeq \liminf_{n \to \infty} A_f[F_n\cdot x] \geq 0$ a.e.

Since $\fsup$ is $T$-invariant, ergodicity implies that it is some constant $C$ almost everywhere. Assume by way of contradiction that $C>c > 0$. Define $\ell_i : X \to \N$ by $x \mapsto$ the $i^{\text{th}}$ $n\in\N$ such that $A_f[F_n x] > c$ (equivalently, $A_{f-c}[F_n x] > 0$).

Fix $\delta > 0$ small enough so that for any measurable $Y \subseteq X$, $\mu(Y)<\delta$ implies $\int_Y (f-c)\, d\mu > -\frac{c}{3}$, and let $M\in\N$ be large enough so that the set $Y \defeq f^{-1}(-M,\infty)$ has measure at least $1 - \delta$.

The tiling property applied to the function $\ell_i$ with $\e \defeq \frac{1}{2(M+c)} \frac{c}{3}$ gives a finite $T\subseteq \Gamma$ such that $\mu(Z) \ge 1 - \e$, where $Z$ is the set of all $x\in X$ such that at least $1-\e$ fraction of $T\cdot x$ is partitioned into sets of the form $\ell_i(y)\cdot y$.

\begin{claim*}
For each $x \in Z$, $A_{\mathbb{1}_Y(f-c)}[T\cdot x] \ge -(M+c) \e$.
\end{claim*}

\begin{pf}
By the definition of $Z$, on a subset $B \subseteq T\cdot x$ that occupies at least $1-\e$ fraction of $T\cdot x$, the average of $f - c$ is positive, and hence that of $\mathbb{1}_Y (f-c) \rest{B}$ is non-negative. On the remaining set $T\cdot x \setminus B$, the function $\mathbb{1}_Y (f-c)$ is at least $-(M+c)$, by the definition of $Y$. Thus, the average of $\mathbb{1}_Y (f-c)$ on the entire $T\cdot x$ is at least $- (M+c) \e$.
\end{pf}
 
\noindent Now we compute using this claim and the invariance of $\mu$:
\begin{align*}
    \int_Y(f-c) \, d\mu  
    &= 
    \int_X A_{\mathbb{1}_Y(f-c)}[T\cdot x]\,d\mu(x)
    \\ 
    &= 
    \int_Z A_{\mathbb{1}_Y(f-c)}[T\cdot x]\,d\mu(x) 
    + 
    \int_{X \setminus Z} A_{\mathbb{1}_Y(f-c)}[T\cdot x] \,d\mu(x) 
    \\
    &\ge 
    -(M+c) \e - (M+c) \e = -2 (M+c) \e = - \frac{c}{3}.
\end{align*}
This gives a contradiction:
\begin{align*}
    0 = \int_X f\, d\mu 
    &= 
    c +\int_X (f-c) \, d\mu
    \\
    &= 
    c + \int_Y(f-c) \, d\mu + \int_{X \setminus Y} (f-c) \, d\mu 
    \\
    &> 
    c - \frac{c}{3} - \frac{c}{3} > 0. \qedhere
\end{align*}
\end{proof}
\section{The tiling property for increasing Tempelman F{\o}lner sequences}\label{differentsizes}

In this section, we prove the following:

\begin{lemma}\label{tempelman}
The tiling property holds for pmp actions of amenable groups along increasing Tempelman F{\o}lner sequences.
\end{lemma}

As a corollary, by \cref{lemma_tserunyan}, we obtain \cref{ptwise_ergodic}. In order to prove this lemma, we need a Vitali covering lemma. For the rest of this section, fix an amenable group $\Gamma$ and Tempelman F\o lner sequence $F_i$ with Tempelman constant $C$, standard probability space $(X,\mu)$ on which $\Gamma$ acts in a pmp way, and $\epsilon>0$.

\begin{lemma}
[Vitali covering]\label{Vitali_covering} Given a function $l:X\rightarrow N$ and a finite subset $S\subseteq X$, there exists a set $K$, which is a disjoint union of sets of the form $F_{l(x)}x$, $x\in S$, such that $|K|\geq\frac{1}{C}|S\cup K|$.
\end{lemma}
\begin{proof}
Put $=S=\{x_1,...,x_n\}$, $D_0=K_0=\emptyset$. We will inductively define increasing sets $K_i$ and $D_i$ for $i\leq n$ until $S\setminus D_i=\emptyset$. Assume $S\setminus D_i\neq \emptyset$. Let $t=\max_{x\in S\setminus {D_i}}l(x)$, and let $j$ be least such that $x_j\in S\setminus D_i$ and $l(x_j)=t$. Put $K_{i+1} \defeq K_i\cup F_t x_j$ and $D_{i+1} \defeq D_i\cup F_t^{-1}F_T\cdot x_j$. Iterate this (up to $n$ times) until $S\setminus D_m=\emptyset$ for some $m\leq n$. Put $K\defeq K_m$ and $D\defeq D_m$. 

We claim that the selected $F_T\cdot x_j$ are actually pairwise disjoint. If not, suppose that at the $(i+1)^{\text{th}}$ step there  is some $y\in F_T\cdot x_j\cap K_i$. Then $x_j\in F_t^{-1}y$. But since $y\in K_i$, there is some $t^\prime\geq t$ and $j^\prime$ such that $y\in F_{t^\prime}x_{j^\prime}$. Hence $x_j\in F_t^{-1}F_{t^\prime}x_{j^\prime}\subseteq F_{t^\prime}^{-1}F_{t^\prime}x_{j^\prime}\subseteq D_i$, contradicting our choice of $x_j$. 

So at each step, we add exactly $|F_t|$ elements to $K$ and at most $C|F_t|$ elements to $D$. Hence, $|K|\geq \frac{1}{C}|D|\geq \frac{1}{C}|S\cup K|$ since $S\cup K\subseteq D$. 
\end{proof}

Now we may begin proving \cref{tempelman}. The idea is to break our space into large finite sets. We will tile each of these finite sets with F{\o}lner shapes of various sizes, using progressively smaller F{\o}lner shapes to fill in whatever holes remain after placing the larger F{\o}lner shapes.


\begin{proof}[Proof of \emph{\cref{tempelman}}]
First, fix $r\in\N$ large enough so that $\left(\frac{C-1}{C}\right)^r<\frac{\epsilon}{2}$. We will ultimately pick $r$ many ``good" sizes of tiles for a large fraction of the points in $X$. Fix $r$ many functions $G_i:[0,1]\rightarrow\mathbb{R}$ such that $G_1(x)\geq\frac{2}{C} (x)$ and for $i>1$, $G_i(x)\geq \frac{C-1}{C}G_{i-1}(x)+(i+1)(x)$ where  each $G_i$ is continuous and $G_i(0)=0$. For example, $G_i(x)\defeq\left(\frac{C-1}{C}\right)^ix+\sum_{k=1}^{i+1}kx$ is such a collection of functions. Fix $\alpha$ small enough so that $\beta\leq \alpha\implies G_r(\beta)<\frac{\epsilon}{2}$. Put $\eta\defeq\min(\alpha,\epsilon)$.

For each $p\in\N$, let $C_n^{(p)}\defeq\{x\in X:(\exists i\in\N)\;l_i(x)\in[p,n]\}$. Since the $l_i$ are strictly increasing in $x$ for all $x$, $\bigcup_n C_n^{(p)}=X$. Hence, there is some large enough $p^*$ such that $\mu(C_{p^*}^{(p)})>1-\frac{\eta^2}{r}$. This means that for any $r$-many values $p_j$ ($0\leq j<r$),  $$\mu(\{x\in X:(\forall j< r)\;(\exists i\in\N)\;l_i(x)\in[p_j,p_j^*]\})>1-\eta^2.$$

We define two sequences of natural numbers of length $r$ as follows. Let $L_0$ be large enough so that $\frac{|\partial_{F_{r}} F_n|}{|F_n|}<\eta$ for all $n>L_0$. For $i<r$, define $R_i\defeq L_i^*$, and $L_{i+1}>R_i$ large enough so that 
$$\frac{|\partial_{F_{R_{i}}} F_n|}{|F_n|}<\frac{\eta}{|F_{R_i}|}$$
for all $n\geq L_{i+1}$. Put $\delta\defeq\frac{\eta}{|F_{R_{r-1}}|}$. 
We will think of the $[L_j,R_j]$ as ranges of allowable sizes for out tiles. Finally,
Let $T\subseteq\Gamma$ satisfy $\frac{|\partial_{F_{R_{r-1}}} T|}{|T|}<\delta$. 

Define parial functions $p_i(x)\defeq l_j(x)$ where $j$ is smallest such that $l_j(x)\in[L_i,R_i]$ if such a $j$ exists. Set $P \defeq \set{x\in X:\;(\forall i<r)\;x\in\text{dom}(p_i)}$. Hence, $\mu(P) > 1-\eta^2$, so $\mu(\set{x\in X:A_{\mathbb{1}_P}[T\cdot x]<1-\eta}) \leq \eta$, because otherwise, setting $B\defeq \set{x\in X:A_{\mathbb{1}_P}[T\cdot x]<1-\eta}$, 
 \begin{align*}
     1-\eta^2 < \mu(P) 
     &= 
     \int_X\mathbb{1}_P(x) d\mu(x)
     \\
     [\text{by the invariance of }\mu]
     &=
     \int_X\frac{1}{|T|}\sum_{\gamma\in T} \mathbb{1}_P(\gamma\cdot x)d\mu(x)
     \\
     &= 
     \int_X A_{\mathbb{1}_P}[T\cdot x] \; d\mu(x)
     \\
     &=
     \int_{X\setminus B}A_{\mathbb{1}_P}[T\cdot x] \; d\mu(x) + \int_B A_{\mathbb{1}_P}[T\cdot x] \; d\mu(x)
     \\
     &\leq \mu(X\setminus B)+ \mu(B)(1-\eta)
     \\
     &<
     1-\eta + \eta(1-\eta)
     =
     1 - \eta^2
 \end{align*}

Hence, at least $1-\eta$ fraction of points $x\in X$ have $1-\eta$ fraction of points of $T\cdot x$ lying in $P$. Since $\eta\leq \epsilon$, it now suffices to show that for a point $x$ such that at least $1-\eta$ fraction of $T\cdot x$ is contained in $P$, we can  tile $T\cdot x$ up to $\epsilon$ fraction with tiles of the form $F_{l_i(x)}x$. 

We claim that in $k$ steps, $1\leq k\leq r$, we can tile $T\cdot x$ up to $\left(\frac{C-1}{C}\right)^k+G_k(\eta)$ fraction. As discussed earlier, we will start by tiling with our largest F{\o}lner shapes, i.e. $l_i(x)\in [L_r,R_r]$, and in each step we will move down a size.

In step $k$, apply \cref{Vitali_covering} with $l=p_{r-k}$ and 
$$S_k\defeq\{x\in T\cdot x\cap \text{dom}(p_{r-k}):F_{R_{r-k}}x\text{ is contained in the set of uncovered  points in }T\}$$
so that the constructed set of tiles $K_k$ is contained in the set of uncovered  points in $T$. 

See \cref{step2} for a sketch of the tiling process for $\Z^2$. In the first picture, we place the tiles from $K_1$, and in the second picture we remove strips along the boundaries of $T\cdot x$ as well as $K_1$. We also mark the points $y$ and $z$ for which $p_{r-2}$ is not defined. Applying \cref{Vitali_covering} to $S_2$ (the remaining points), we place smaller tiles, seen in a lighter color in the third picture.
\begin{figure}[htp]
    \centering
    \includegraphics[width=5.3cm]{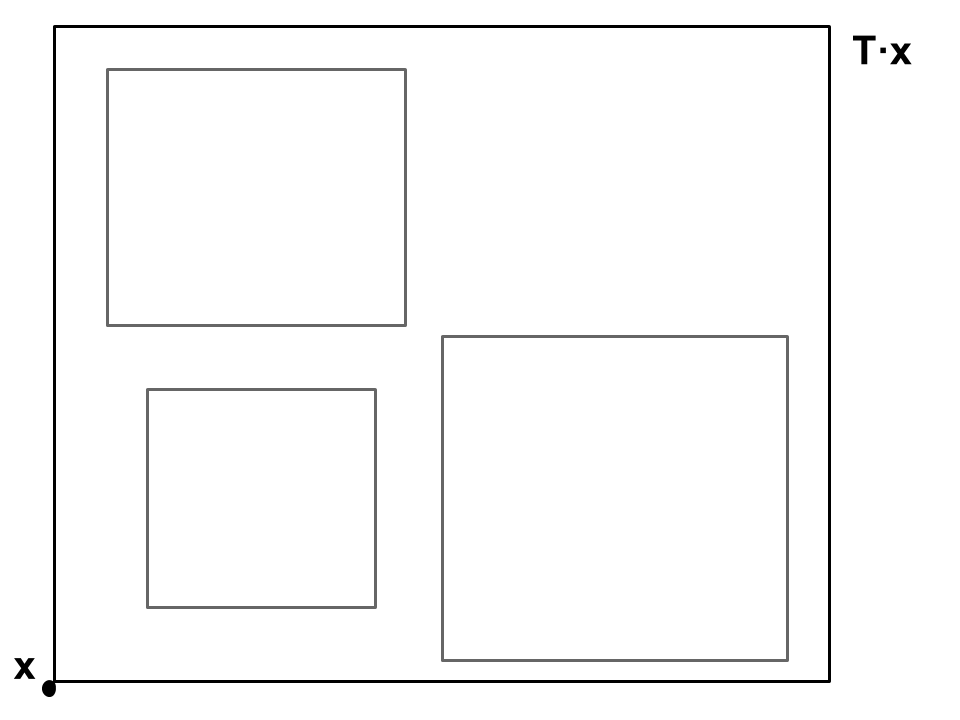}
    \includegraphics[width=5.3cm]{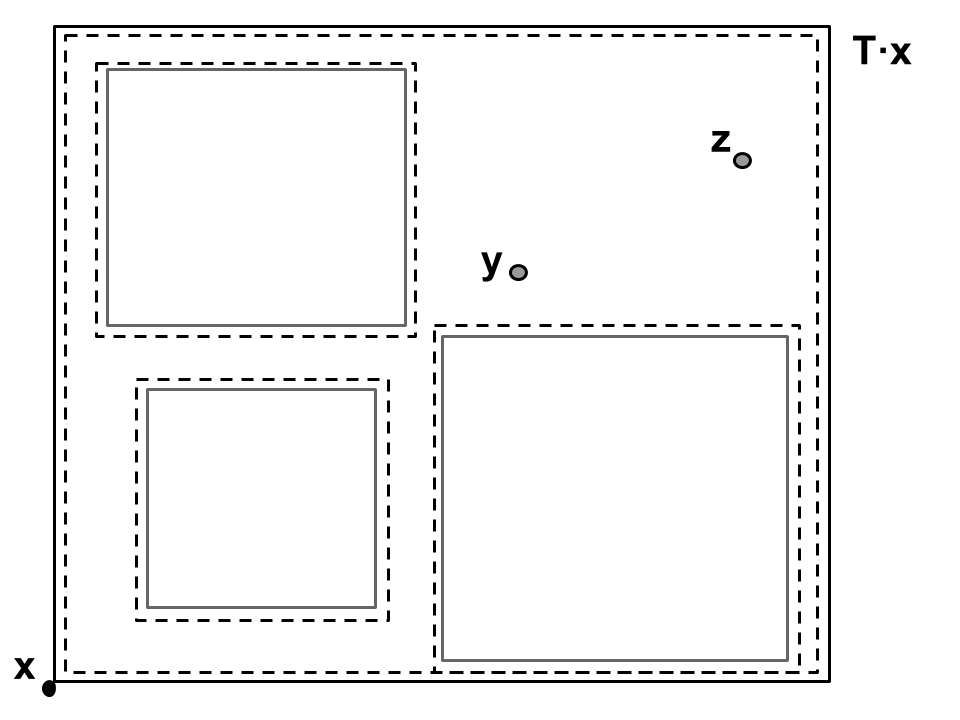}
    \includegraphics[width=5.3cm]{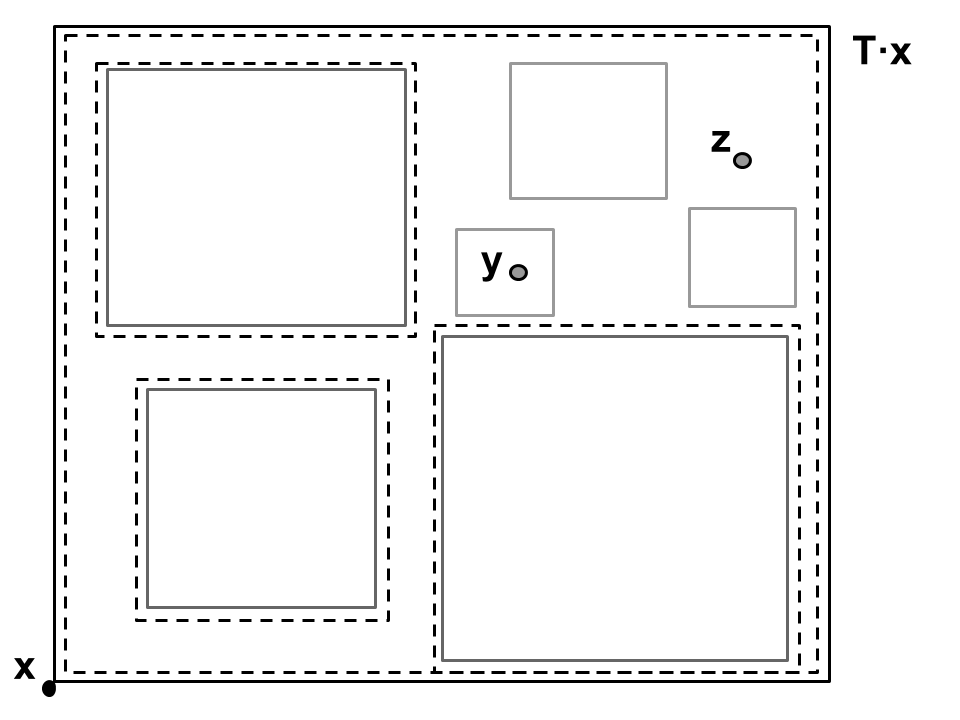}
    \caption{Identifying the set $S_2$ and placing the set of tiles $K_2$ when $\Gamma=\Z^2$}
    \label{step2}
\end{figure}

Now, $S_k$ is almost all of the uncovered points in $T$ except possibly:
\begin{enumerate}
    \item A small strip along the boundary of $T$, of size $|\partial_{F_{R_{r-k}}} T||F_{R_{r-k}}|<\eta|T|$ since $\frac{|\partial_{F_{R_{r-k}}} T|}{|T|}<\frac{\eta}{|F_{R_{r-}k}|}$.
    \item The set of points on which $p_{r-k}$ is not defined, which has fewer than $\eta|T|$ points.
    \item A small strip along the boundary of the covered points from each of the previous $k-1$ steps. Fix $j<k$, and consider the set $K_j$ of covered points from the $j^\text{th}$ step. We might miss a strip of size $|\partial_{F_{R_{r-k}}} K_j||F_{R_{r-k}}|$. Note that since the boundary of $K_j$ consists of F\o lner shapes of sizes in $[L_{r-j},R_{r-j}]$, we have $\frac{|\partial_{F_{R_{r-k}}} K_j|}{|K_j|}\leq \frac{|\partial_{F_{R_{r-j}}} F_{R_{r-j}}|}{|F_{R_{r-j}}|}$, so 
    \begin{align*}
        |\partial_{F_{R_{r-k}}} K_j||F_{R_{r-k}}|&\leq \frac{|\partial_{F_{R_{r-j}}} F_{R_{r-j}}|}{|F_{R_{r-j}}|}|K_j||F_{R_{r-k}}| \\
        &\leq \frac{\eta}{|F_{R_{r-j-1}}|}|K_j||F_{R_{r-k}}|\\
        &\leq \eta|T|,
    \end{align*}
    where the penultimate inequality comes from our choice of $L_{r-j}$ to be large enough that $n\geq L_{r-j}$ implies $\frac{|\partial_{F_{R_{r-j}}} F_n|}{|F_n|}<\frac{\eta}{|F_{R_{r-j-1}}|}$, and the final inequality comes from $K_j\subseteq T$ and the fact that $r-k\leq r-j-1$ for any $j<k$.
\end{enumerate}

In total, $S_k$ is missing at most $(k+1)\eta|T|$ uncovered  points from $T$. If $k=1$, we have that $K_1$ covers at least $\frac{1}{C}$ fraction of $S_1\cup K_1$, and $|S_1\cup K_1|\geq (1-2\eta)|T|$. So $K_1$ covers at least $\frac{1}{C}(1-2\eta)|T|$, and we are left with $\frac{C-1+2\eta}{C}|T|$, so we cover all but $\frac{C-1}{C}+G_1(\eta)$ fraction of $|T|$.

If $k\geq 2$, assume $\bigcup_{i<k}K_i$ covers all but $\left(\frac{C-1}{C}\right)^{k-1}+G_{k-1}(\eta)$ fraction of $T$. Notice that 
$$|S_k\cup K_k|\leq \left(\left(\frac{C-1}{C}\right)^{k-1}+G_{k-1}(\eta)\right)|T|,$$
since both $S_k$ and $K_k$ are contained in the set of uncovered points of $T$. Since $K_k$ covers at least $\frac{1}{C}$ fraction of $|S_k\cup K_k|$, at most $\frac{C-1}{C}$ fraction of $|S_k\cup K_k|$ is left uncovered. So $\bigcup_{i\leq k}K_k$ covers all of $T$ but at most
\begin{align*}
    (k+1)\eta|T|+\frac{C-1}{C}|S_k\cup K_k|&\leq \left((k+1)\eta+\left(\frac{C-1}{C}\right)^{k}+\frac{C-1}{C}G_{k-1}(\eta)\right)|T|\\
    &\leq \left(\left(\frac{C-1}{C}\right)^{k}+G_k(\eta)\right)|T|
\end{align*}
many points. This concludes the proof of our claim. Iterate this algorithm $r$ times so that we have covered all but $\left(\frac{C-1}{C}\right)^r+G_r(\eta)$ fraction of $T$. Since, by hypothesis, both $\left(\frac{C-1}{C}\right)^r, G_r(\eta)<\frac{\epsilon}{2}$, this concludes the proof.
\end{proof}

\end{document}